\documentclass{amsart}
\usepackage{url} 
\usepackage{amsmath,amssymb,amsthm}
\usepackage{verbatim}
\usepackage{enumerate}

\DeclareMathOperator{\gal}{Gal}
\DeclareMathOperator{\ord}{ord}

\DeclareMathOperator{\aut}{Aut}

\DeclareMathOperator{\rad}{rad}

\newtheorem{thm}{Theorem}[section]
\newtheorem{lem}[thm]{Lemma}
\newtheorem{cor}[thm]{Corollary}

\newtheorem{prop}[thm]{Proposition}
\newtheorem{exa}[thm]{Example}
\newtheorem{rem}[thm]{Remark}

\begin{document}
\title{Perfect powers generated by the twisted Fermat cubic}
\author{Jonathan Reynolds}

\address{Mathematisch Instituut \\ Universiteit Utrecht \\ Postbus 80.010 \\ 3508 TA Utrecht \\ Nederland}
\email{J.M.Reynolds@uu.nl}
\date{\today}
\thanks{The author is supported by a Marie Curie Intra European Fellowship (PIEF-GA-2009-235210)}

\begin{abstract}
On the twisted Fermat cubic, an elliptic divisibility sequence arises as the sequence of denominators of the multiples of a single rational point. It is shown that there are finitely many perfect powers in such a sequence whose first term is greater than $1$. Moreover, if the first term is divisible by $6$ and the generating point is triple another rational point then there are no perfect powers in the sequence except possibly an $l$th power for some $l$ dividing the order of $2$ in the first term.                               
\end{abstract} 

\maketitle

\section{Introduction}

A divisibility sequence is a sequence 
\[
W_1, W_2, W_3, \ldots
\] 
of integers satisfying $W_n|W_m$ whenever $n|m$. The arithmetic of these has been and continues to be of great interest. Ward \cite{MR0023275} studied a large class of recursive divisibility sequences and gave equations for points and curves from which they can be generated (see also \cite{Shi01}). In particular, Lucas sequences can be generated from curves of genus $0$. Although Ward did not make such a distinction, sequences generated by curves of genus $1$ have become exclusively known as elliptic divisibility sequences \cite{MR2164113, MR2045409, MR2301226, IngrSilv} and have applications in Logic \cite{MR2377127, MR2480276, EES} as well as Cryptography
\cite{StangeL}. See \cite{MR2514094, MR1171452} for background on elliptic curves (genus-$1$ curves with a point).
Let $d \in \mathbb{Z}$ be cube-free and consider the elliptic curve 
\[
C: u^3+v^3=d.
\]  
It is sometimes said that $C$ is a twist of the Fermat cubic.
The set $C(\mathbb{Q})$ forms a group under the
chord and tangent method: the (projective) point $[1, -1, 0]$ is the identity and inversion is given by reflection in the line $u=v$.  Suppose that $C(\mathbb{Q})$ contains a non-torsion point $P$. Then we can write, in lowest terms,
\begin{equation}
mP=\left( \frac{U_m}{W_m}, \frac{V_m}{W_m} \right).
\end{equation} 
The sequence $(W_m)$ is a (strong) divisibility sequence (see Proposition 3.3 in \cite{EvOm}). Three particular questions about divisibility sequences have received much interest:
\begin{itemize}
\item How many terms fail to have a primitive divisor?
\item How many terms are prime?
\item How many terms are a perfect power?
\end{itemize}
A primitive divisor is a prime divisor which does not divide any previous term.

\subsection{Finiteness} Bilu, Hanrot and Voutier proved that all terms in a Lucas sequence beyond the $30$th have a primitive divisor \cite{MR1863855}.
Silverman showed that finitely many terms in an elliptic divisibility sequence fail to have to have a primitive divisor~\cite{MR961918} (see also \cite{Marcothesis}). The Fibonacci and Mersenne sequences are believed to have infinitely many prime terms \cite{CaldM, CaldF}. The latter has produced the largest primes known to date. In \cite{MR866702} Chudnovsky and Chudnovsky
considered the likelihood that an elliptic divisibility sequence might be a source of large primes; however, $(W_m)$ has been shown to contain only finitely many prime terms \cite{MR2045409}. Gezer and Bizim have described the squares in some periodic divisibility sequences \cite{MR2669714}. Using modular techniques inspired by the proof of Fermat's Last Theorem, it was finally shown in \cite{MR2215137} that the only perfect powers in the Fibonnaci sequence are $1$, $8$ and $144$. We will show:       
  
\begin{thm} \label{finite}   
If $W_1>1$ then there are finitely many perfect powers in $(W_m)$.
\end{thm}

The proof of Theorem \ref{finite} uses the divisibility properties of $(W_m)$ along with a modular method for cubic binary forms given in \cite{MR2406491}. For elliptic curves in Weierstrass form similar results have been shown in \cite{Rey02}. In the general case, allowing for integral points, Conjecture 1.1 in \cite{MR2406491} would give that there are finitely many perfect powers in $(W_m)$. 

\subsection{Uniformness}
What is particularly special about sequences $(W_m)$ coming from twisted Fermat cubics is that they have yielded uniform results as sharp as some of their genus-$0$ analogues mentioned above. It has been shown that all terms of 
$(W_m)$ beyond the first have a primitive divisor \cite{MR2486632} and, in particular, we will make use of the fact that the second term always has a primitive divisor $p_0>3$ (see Section 6.2 in \cite{MR2486632}). The number of prime terms in $(W_m)$ is also uniformly bounded~\cite{EvOm} and, in particular, if $P$ is triple a rational point then all terms beyond the first fail to be prime (see Theorem 1.2 in \cite{EvOm}). In light of Theorem \ref{finite}, it is natural to ask if a similar results can be achieved for perfect powers. Indeed:   

\begin{thm} \label{uniform} Suppose that $W_1$ is even and at all primes greater than $3$, $P$ has non-singular reduction (on a minimal Weierstrass equation for $C$). 
If $W_m$ is an $l$th power for some prime $l$ then 
\[
l \le \max \left\{ \ord_2(W_1), (1+\sqrt{p_0})^2  \right\}, 
\]
where $p_0 > 3$ is a primitive divisor of $W_2$.
Moreover,  for fixed $l>\ord_2(W_1)$ the number of $l$th powers in $(W_m)$ is uniformly bounded.
\end{thm}

Although the conditions in Theorem \ref{uniform} appear to depend heavily on the point, in the next theorem we exploit the fact that group $C(\mathbb{Q})$ modulo the points of non-singular reduction has order at most $3$ for a prime greater than $3$.

\begin{thm} \label{uniform2}
Suppose that $6 \mid W_1$ and $P \in 3E(\mathbb{Q})$ (or $P$ has non-singular reduction at all primes greater than $3$). If $W_m$ is an 
$l$th power for some prime $l$ then $l \mid \ord_2(W_1)$. In particular, if 
$\ord_2(W_1)=1$ then $(W_m)$ contains no perfect powers.
\end{thm}

The conditions in Theorem \ref{uniform2} are sometimes satisfied for every rational non-torsion point on $C$. For example, we have

\begin{cor} The only solutions to the Diophantine equation 
\[
U^3+V^3=15W^{3l}
\] 
with $l>1$ and $\gcd(U,V,W)=1$ have $W=0$.
\end{cor}

\section{Properties of elliptic divisibility sequences}

In this section the required properties of $(W_m)$ are collected.

\begin{lem} \label{wedsp} Let $p$ be a prime. For any pair $n,m \in \mathbb{N}$, if $\ord_p(W_n)>0$ then
\[
\ord_p(W_{mn})=\ord_p(W_n)+\ord_p(m).
\]
\end{lem}

\begin{proof}
See equation (10) in \cite{EvOm}. 
\end{proof}

\begin{prop} \label{strongw} For all $n,m \in \mathbb{N}$,
\[
\gcd(W_m, W_n)=W_{\gcd(m,n)}.
\]
In particular, for all $n,m \in \mathbb{N}$, $W_n \mid W_{nm}$. 
\end{prop}

\begin{proof}
See Proposition 3.3 in \cite{EvOm}.
\end{proof}

\begin{thm}[\cite{MR2486632}] \label{prim1} If $m>1$ then $W_m$ has a primitive divisor.
\end{thm}

\section{The modular approach to Diophantine equations}

For a more thorough exploration see \cite{Sanderthesis} and Chapter 15 in \cite{MR2312338}. As is conventional, in what follows all newforms shall have weight $2$ with a trivial character at some level $N$ and shall be thought of as a $q$-expansion
\[
f=q+\sum_{n \ge 2}c_nq^n, 
\]
where the field $K_f=\mathbb{Q}(c_2,c_3,\cdots)$ is a totally real number field. The coefficients $c_n$ are 
algebraic integers and $f$ is called \emph{rational} if they all belong to $\mathbb{Z}$. For a given level $N$, the number of newforms is finite. The modular symbols algorithm \cite{MR1628193}, implemented on $\mathtt{MAGMA}$ \cite{MR1484478} by William Stein, shall be used to compute the newforms at a given level. 

\begin{thm}[Modularity Theorem]
Let $E/\mathbb{Q}$ be an elliptic curve of conductor $N$. Then there exists a newform $f$ of level $N$ such that $a_p(E)=c_p$ for all primes 
$p \nmid N$, where $c_p$ is $p$th coefficient of $f$ and 
$a_p(E)=p+1-\#E(\mathbb{F}_p)$.  
\end{thm}

\begin{proof}
This is due to Taylor and Wiles \cite{MR1333036, MR1333035} in the semi-stable case. The proof was completed by Breuil, Conrad, Diamond and Taylor \cite{MR1839918}.     
\end{proof}

The modularity of elliptic curves over $\mathbb{Q}$ can be seen as a converse to   

\begin{thm}[Eichler-Shimura] Let $f$ be a rational newform of level $N$. There exists an elliptic curve $E/\mathbb{Q}$ of conductor $N$ such that 
$a_p(E)=c_p$ for all primes $p \nmid N$, where $c_p$ is the $p$th coefficient of $f$ and $a_p(E)=p+1-\#E(\mathbb{F}_p)$.
\end{thm}

\begin{proof}
See Chapter $8$ of \cite{MR2112196}.
\end{proof}

Given a rational newform of level $N$, the elliptic curves of conductor $N$ associated to it via the Eichler-Shimura theorem shall be computed using $\mathtt{MAGMA}$.  

\begin{prop} \label{levellow} Let $E/\mathbb{Q}$ be an elliptic curve with conductor $N$ and minimal discriminant $\Delta_{\min}$.
Let $l$ be an odd prime and define
\[ 
N_0(E,l):=N/\mathop{\prod_{{\textrm{primes } p \mid \mid N}}}_{l \mid \ord_p(\Delta_{\min})} p. 
\]
Suppose that the Galois representation
\[ 
\rho_l^E: \gal(\bar{\mathbb{Q}}/\mathbb{Q}) \to \aut(E[l])
\]
is irreducible. Then there exists a newform $f$ of level $N_0(E,l)$. Also there exists  a prime $\mathcal{L}$ lying above $l$ in the ring of integers $\mathcal{O}_f$ defined by the coefficients of $f$ such that  
\[
c_p \equiv \left\{ \begin{array}{ll} a_p(E)  \mod \mathcal{L} & \textrm{ if } p \nmid lN, \\ \pm(1+p) \mod \mathcal{L} & \textrm{ if } p \mid \mid N \textrm{ and } p \nmid lN_0, \end{array} \right.
\]
where $c_p$ is the $p$th coefficient of $f$. Furthermore, if $\mathcal{O}_f=\mathbb{Z}$ then
\[
c_p \equiv \left\{ \begin{array}{ll} a_p(E)  \mod l & \textrm{ if } p \nmid N, \\ \pm(1+p) \mod l & \textrm{ if } p \mid \mid N \textrm{ and } p \nmid N_0. \end{array} \right.
\]    
\end{prop}

\begin{proof}
This arose from combining modularity with level-lowering results by Ribet~\cite{MR1047143, MR1265566}. The strengthening in the case $\mathcal{O}_f=\mathbb{Z}$ is due to Kraus and Oesterl{\'e}~\cite{MR1166121}.   
A detailed exploration is given, for example, in Chapter 2 of \cite{Sanderthesis}. 
\end{proof}

\begin{rem} \rm Let $E/\mathbb{Q}$ be an elliptic curve with conductor $N$.
Note that the exponents of the primes in the factorization of $N$ are uniformly bounded (see Section 10 in Chapter IV of \cite{MR1312368}). In particular, only primes of bad reduction divide $N$ and if $E$ has multiplicative reduction at $p$ then $p \mid \mid N$. 
\end{rem}

\begin{cor} \label{lbound} Keeping the notation of Proposition \ref{levellow}, if $p$ is a prime such that $p \nmid lN_0$ and  $p \mid N$ then
\[
l < (1+\sqrt{p})^{2[K_f: \mathbb{Q}]}.
\]    
\end{cor}

\begin{proof}
See Theorem 37 in \cite{Sanderthesis}.
\end{proof}

Applying Proposition \ref{levellow} to carefully constructed Frey curves has led to the solution of many Diophantine problems. The most famous of these is Fermat's Last theorem \cite{MR1333035} but there are now constructions for other equations and we shall make use of those described below.

\subsection{A Frey curve for cubic binary forms} \label{cbf} Let
\[
F(x,y)=t_0a^3+t_1^2y+t_2xy^2+t_3y^3 \in \mathbb{Z}[x,y]
\]
be a separable cubic binary form. 
In \cite{MR2406491} a Frey curve is given for the Diophantine equation
\begin{equation} \label{cub}
F(a,b)=dc^l,
\end{equation}
where $\gcd(a,b)=1$, $d \in \mathbb{Z}$ is fixed and $l \ge 7$ is prime. 
Define a Frey curve $E_{a,b}$ by
\begin{equation} \label{BillFrey}
E_{a,b}: y^2=x^3+a_2x^2+a_4x+a_6,
\end{equation}
where
\begin{eqnarray*}
a_2 &=& t_1a -t_2b, \\
a_4 &=& t_0t_2a^2 +(3t_0t_3 -t_1t_2)ab + t_1t_3b^2, \\
a_6 &=& t_0^2t_3a^3-t_0(t_2^2-2t_1t_3)a^2b+t_3(t_1^2-2t_0t_2)ab^2-t_0t_3^2b^3.
\end{eqnarray*}
Then $E_{a,b}$ has discriminant $16 \Delta_F F(a,b)^2$. Consider the Galois representation
\[
\rho_l^{a,b} : \gal(\bar{\mathbb{Q}}/\mathbb{Q}) \to \aut(E_{a,b}[l]).
\]
\begin{thm}[\cite{MR2406491}] \label{Bill} Let $S$ be the set of primes dividing $2d\Delta_F$. 
There exists a constant $\alpha(d,F) \ge 0$ such that if $l>\alpha(d,F)$ 
and $c \ne \pm 1$ then:
\begin{itemize}
\item the representation $\rho_l^{a,b}$ is irreducible;
\item at any prime $p \notin S$ dividing $F(a,b)$ the equation (\ref{BillFrey}) is minimal, the elliptic curve $E_{a,b}$ has multiplicative reduction and $l \mid \ord_p(\Delta_{min}(E_{a,b}))$.   
\end{itemize} 
\end{thm}

\subsection{Recipes for Diophantine equations with signature $(l,l,l)$} \label{rec}
The following recipe due to Kraus \cite{MR1611640} is taken from \cite{MR2312338}. Consider the equation
\[
Ax^l+By^l+Cz^l=0,
\]
with non-zero pairwise coprime terms and $l \ge 5$ prime. Setting $R=ABC$ assume that any prime $q$ satisfies $\ord_q(R)<l$. Without lost of generality also assume that $By^l \equiv 0 \mod 2$ and $Ax^l \equiv -1 \mod 4$.  
Construct the Frey curve
\[
E_{x,y}: Y^2=X(X-Ax^l)(X+By^l). 
\]
The conductor $N_{x,y}$ of $E_{x,y}$ is given by
\[
N_{x,y}=2^{\alpha}\rad_2(Rxyz),
\]
where
\[
\alpha= \left\{ \begin{array}{ll} 1 & \textrm{if } \textrm{ } \ord_2(R) \ge 5 
\textrm{ or } \ord_2(R)=0, \\
1 & \textrm{if } \textrm{ } 1 \le \ord_2(R) \le 4 \textrm{ and } y \textrm{ is even}, \\
0 & \textrm{if } \textrm{ } \ord_2(R)=4 \textrm{ and } y \textrm{ is odd}, \\
3 & \textrm{if } \textrm{ } 2 \le \ord_2(R) \le 3 \textrm{ and } y \textrm{ is odd}, \\
5 & \textrm{if } \textrm{ } \ord_2(R)=1 \textrm{ and } y \textrm{ is odd}.
\end{array} \right.  
\]
\begin{thm}[Kraus \cite{MR1611640}] \label{reci}
The Galois representation
\[
\rho_l^{x,y} : \gal(\bar{\mathbb{Q}}/\mathbb{Q}) \to \aut(E_{x,y}[l])
\]
is irreducible and $N_0(E_{x,y},l)$ in Proposition \ref{levellow} is given by   
\[
N_0=2^{\beta} \rad_2(R), 
\]
where
\[
\beta= \left\{ \begin{array}{ll} 1 & \textrm{if } \textrm{ } \ord_2(R) \ge 5 
\textrm{ or } \ord_2(R)=0, \\
0 & \textrm{if } \ord_2(R)=4, \\
1 & \textrm{if } \textrm{ } 1 \le \ord_2(R) \le 3 \textrm{ and } y \textrm{ is even}, \\
3 & \textrm{if } \textrm{ } 2 \le \ord_2(R) \le 3 \textrm{ and } y \textrm{ is odd}, \\
5 & \textrm{if } \textrm{ } \ord_2(R)=1 \textrm{ and } y \textrm{ is odd}.
\end{array} \right.  
\]
\end{thm}

\section{Proof of Theorem \ref{finite}}

\begin{proof}[Proof of Theorem \ref{finite}] 
Assume that $W_1>1$ and $W_m$ is an $l$th power for some prime $l$.  
Firstly we will use the Frey curve for cubic binary forms constructed in Section \ref{cbf} and prove the existence of a prime divisor $p$ to which Corollary \ref{lbound} can be applied, giving a bound for $l$.   
Let $S$ be the set of primes dividing $27d$. By assumption, $W_1$ is divisible by a prime $q$.  Lemma~\ref{wedsp} gives that
\[
l \le \ord_{q}(W_m)=\ord_{q}(W_1)+\ord_{q}(m).
\]
Using Theorem \ref{prim1} (or that there are only finitely many solutions to a Thue-Mahler equation), let $l$ be large enough so that $W_n$ is divisible by a prime $p \notin S$, where 
\[
n=q^{l-\ord_{q}(W_1)}.
\]
Note that we can choose this lower bound for $l$ and $p$ independently of $m$. 
Then, using Proposition \ref{strongw}, $p \mid W_m$. Now construct a Frey curve $E_{U,V}$ for the Diophantine equation
\[
U_m^3+V_m^3=dW^l
\]  
as in Section \ref{cbf} (in our case $F(x,y)=x^3+y^3$) and consider the Galois representation
\[
\rho_l : \gal(\bar{\mathbb{Q}}/\mathbb{Q}) \to \aut(E_{U,V}[l]).
\]
Using Theorem \ref{Bill}, choose $l$ larger than some constant so that $p$ divides the conductor of $E_{U,V}$ exactly once and the primes dividing $N_0$ in Proposition \ref{levellow} belong to $S$. Since there are finitely many newforms of level $N_0$, Corollary \ref{lbound} bounds $l$.
Finally, for fixed $l$ there are finitely many solutions by Theorem 1 in \cite{MR1348707}. \end{proof} 

\section{Proof of Theorem \ref{uniform}}

\begin{proof}[Proof of Theorem \ref{uniform}] Assume that $W_m$ is an $l$th power. We will derive an $(l,l,l)$ equation (\ref{this}) which does not depend on $d$ and use the Frey curve given Section \ref{rec}. Then, similarly to the proof of Theorem \ref{finite}, the existence of a prime divisor $p_0$ will be shown which bounds $l$ via Corollary \ref{lbound}. Since $2 \mid W_1$, by Lemma~\ref{wedsp},
\[
l \le \ord_2(W_{m})=\ord_2(W_1)+\ord_2(m).
\]
Assume that $l>\ord_2(W_1)$. Then $\ord_2(m)>0$ so $m=2m'$ for some $m'$.

A Weierstrass equation for $C$ is
\begin{equation} \label{we}
y^2=x^3-2^43^3d^2,
\end{equation}
with coordinates $x=2^23d/(u+v)$ and $y=2^23^2d(u-v)/(u+v)$. Write 
$x(mP)=A_m/B_m^2$ and $y(mP)=C_m/B_m^3$ in lowest terms.

\begin{lem}[see Corollary 3.2 in \cite{EvOm}] \label{23}
Let $p=2$ or $3$.  then $p \mid W_m$ if and only if $p \nmid A_m$. 
\end{lem}

The discriminant of (\ref{we}) is $-2^{12}3^9d^4$ so, since $d$ is cube free, it is minimal at any prime larger than $3$ (see Remark 1.1 in Chapter VII \cite{MR2514094}). Note that the group of points with non-singular reduction is independent of the choice of minimal Weierstrass equation. The projective equation of (\ref{we}) is
\[
Y^2Z=X^3-2^43^3d^2Z^3.
\]
Let $p > 3$ be a prime dividing $d$.  By assumption, the partial derivatives
\begin{equation} \label{pds}
\frac{\partial C}{\partial X} = -3X^2, \textrm{ } 
\frac{\partial C}{\partial Y}=2YZ \textrm{ } \textrm{ and } \textrm{ } 
\frac{\partial C}{\partial Z}=Y^2+2^43^4d^2Z^2
\end{equation}
do not vanish simultaneously at $P=[A_1B_1,C_1,B_1^3]$ over the field $\mathbb{F}_p$. Hence, noting that $2 \nmid A_m$ from Lemma \ref{23} and that non-singular points form a group, we have
\begin{equation} \label{cancel}
\gcd(A_{m}^3,C_{m}^2) \mid 3^{3+2\ord_3(d)} 
\end{equation}
for all $m$.

The inverses of the birational transformation are given by 
$u=(2^23^2d+y)/6x$ and $v=(2^23^2d-y)/6x$.
Thus
\begin{equation} \label{uv}
\frac{U_m}{W_m}=\frac{2^23^2dB_{m}^3+C_{m}}{6A_{m}B_{m}} \textrm{ } \textrm{ and } \textrm{ } \frac{V_m}{W_m}=\frac{2^23^2dB_{m}^3-C_{m}}{6A_{m}B_{m}}. 
\end{equation}
The assumptions made restrict the cancellation which can occur in (\ref{uv}) and, up to cancellation, if $W_m$ is an $l$th power then so is $A_m$.   
More precisely, since $W_m$ is an $l$th power and $2 \mid W_m$, Lemma \ref{23} and (\ref{cancel}) give that $A_{m}$ is an $l$th power multiplied by a power of $3$.   Using the duplication formula,
\begin{equation} \label{2Q2}
\frac{A_m}{B_m^2}=\frac{A_{m'}(A_{m'}^3+8(2^43^3d^2)B_{m'}^6)}{4B_{m'}^2(A_{m'}^3-2^43^3d^2B_{m'}^6)}=\frac{A_{m'}(A_{m'}^3+8(2^43^3d^2)B_{m'}^6)}{4B_{m'}^2C_{m'}^2}.
\end{equation}
Again, cancellation in (\ref{2Q2}) is restricted so $A_{m'}$ is also an $l$ power multiplied by a power of $3$. Write 
\[
m=2^{\ord_2(m)}n.
\] 
It follows that   
$A_{n}= 3^eA^l$,
\[
A_n^3+8(2^43^3d^2)B_n^6 = 3^f \bar{A}^l
\]
and $C_n= \pm 3^gC^l$. Combining with $C_n^2=A_n^3-2^43^3d^2B_n^6$ gives
\begin{equation} \label{this}
3^f\bar{A}^l + 2^33^{2g}C^{2l}=3^{2+3e}A^{3l}.
\end{equation}
Note that, by dividing (\ref{this}) through by an appropriate power of $3$, we can assume that $3$ divides at most one of the three terms. 

Let $p_0>3$ be a primitive divisor of $W_2$. Using Proposition \ref{strongw}, $p_0 \mid W_{2n}$ and, 
since $n$ is odd, $p_0 \mid \bar{A} C$. Now follow the recipe given in Section \ref{rec}. The conductor of the Frey curve for (\ref{this}) is
\[
N_{\bar{A}, C}=2^3 3^{\delta}\rad_3(\bar{A}CA)
\]
and $N_0=2^33^{\delta}$ in Theorem \ref{reci}, where $\delta=0$ or $1$. There is one newform
\[
f=q-q^3-2q^5+q^9+4q^{11}+\cdots
\]
of level $N_0=24$. Moreover, $f$ is rational. Since $p_0 \mid N_{\bar{A}, C}$ and $p_0 \nmid N_0$, 
\[
l < (1+\sqrt{p_0})^2
\]
by  Corollary~\ref{lbound}. Finally, for fixed $l>1$ there are finitely many solutions to (\ref{this}) (see Theorem $2$ in \cite{MR1348707}) and they are independent of $d$.  
\end{proof}   
    
\section{Proof of Theorem \ref{uniform2}}
\begin{proof}[Proof of Theorem \ref{uniform2}]
As in the proof of Theorem \ref{uniform}, consider 
$x(P)=A_P/B_P^2$ and $y(P)=C_P/B_P^3$ on the Weierstrass equation
\[
y^2=x^3-2^43^3d^2
\]
for $C$. Since $P$ is triple another rational point, a prime of bad reduction greater $3$ does not divide $A_P$ (see Section 3 in \cite{MR2486632}).   
Thus the partial derivatives (\ref{pds}) do not vanish simultaneously at $P$ and so at all primes greater than $3$, $P$ has non-singular reduction on a minimal Weierstrass for $C$.  

Now follow the proof of Theorem \ref{uniform} up to (\ref{2Q2}). Factorizing over 
$\mathbb{Z}[\sqrt{-3}]$ gives
\[  
A_n^3=C_n^2+2^43^3d^2B_n^6=(C_n+2^23dB_n^3\sqrt{-3})(C_n-2^23dB_n^3\sqrt{-3}).
\]
We have
\[
C_n+2^23dB_n^3\sqrt{-3}=(-1+\sqrt{-3})^s(a+b\sqrt{-3})^3/2^{s+3},
\]
where $s=0,1$ or $2$ and $a,b$ are integers of the same parity.
If $s=0$ then
\[
2^3(C_n+2^23dB_n^3\sqrt{-3})=a(a^2-9b^2)+3b(a^2-b^2)\sqrt{-3},
\]
so
\begin{eqnarray}
2^3C_n&=&a(a^2-9b^2), \label{h1} \\
2^5dB_n^3&=&b(a^2-b^2), \label{h2} \\
2^2A_n&=&a^2+3b^2. \label{h3}
\end{eqnarray}
If $s=1$ then
\begin{eqnarray*}
2^4C_n&=&-a^3+9ab^2-9a^2b+9b^3, \\
2^63dB_n^3&=&a^3-3a^2b-9ab^2+3b^3,\\
2^2A_n&=&a^2+3b^2.
\end{eqnarray*}
If $s=2$ then
\begin{eqnarray*}
2^5C_n&=&-2a^3+18a^2b+18ab^2-18b^3, \\
2^73dB_n^3&=&-2a^3-6a^2b+18ab^2+6b^3, \\
2^2A_n&=&a^2+3b^2. 
\end{eqnarray*}
By Lemma \ref{23}, $6 \nmid A_n$ so we are in the case $s=0$. 

Suppose that $W_m$ is a square. Then, from (\ref{2Q2}), $C_n=\pm C^2$, $2B_n=\pm B^2$ and $A_n=A^2$. 
Since $\gcd(a,b) \mid 2^2$, one of $b$ or $a^2-b^2$ is coprime with the odd primes dividing $d$. If it is $b$ then multiplying (\ref{h1}) and (\ref{h3}) gives
\[
\pm 2^5 (AC)^2=a^5 - 6a^3b^2 - 27ab^4
\]
and, since $b$, up to sign, is either a square or $2$ multiplied by a square, dividing by $b^5$ gives a rational point on the hyperelliptic curve
\[
Y^2=X^5-6X^3-27X
\]
with non-zero coordinates; but computations implemented in $\mathtt{MAGMA}$ confirm that the Jacobian of the curve has rank $0$ and, via the method of Chabauty, there are no such points. If $a^2-b^2$ is coprime with the odd primes dividing $d$ then multiplying with (\ref{h3}) gives a rational point on the elliptic curve
\[
\pm Y^2=X^4+2X^2-3
\]
or on the elliptic curve
\[
\pm 2^3Y^2=X^4+2X^2-3
\]
with non-zero coordinates; but there are no such points.

Suppose that $W_m$ is an $l$th power for some odd prime $l$. Then, from (\ref{2Q2}), $C_n$, $2B_n$ and $A_n$ are $l$th powers. If $a$ is odd then (\ref{h1}) gives $a=C^l$, $a^2-9b^2=2^3\bar{C}^l$ and
\begin{equation} \label{t1}
C^{2l}-2^3\bar{C}^l=9b^2.
\end{equation}
If $a$ is even then $a=2C^l$, $a^2-9b^2=2^2\bar{C}^l$ and 
\begin{equation} \label{t2}
2^2C^{2l}-2^2\bar{C}^l=9b^2.
\end{equation}
Thus, Theorem 15.3.4 in \cite{MR2312338} (due to Bennett and Skinner \cite{MR2031121}, Ivorra \cite{MR1979902} and Siksek~\cite{MR2142239}) and Theorem 15.3.5 in \cite{MR2312338} (due to Darmon and Merel \cite{MR1468926}) give that $l \le 5$. If $l=3$ then we have a rational point on the elliptic curve
\[
Z^6+X^3=Y^2;
\]
this curve has rank and gives a possible
solution $\bar{C}=-1$, $a=C = \pm 1$ and $b=\pm 1$, but, 
from (\ref{h2}), we would have $B_n=0$. If $l=5$ then we have a rational point on the hyper elliptic curve
\[
Y^2=8^eX^5+1,
\]
where $e=0$ or $1$; but computations implemented in $\mathtt{MAGMA}$ confirm, via the method of Chabauty, that no such points give a required solution. 
\end{proof}
 
\bibliographystyle{amsplain}
\bibliography{myrefs}

\providecommand{\bysame}{\leavevmode\hbox to3em{\hrulefill}\thinspace}
\providecommand{\MR}{\relax\ifhmode\unskip\space\fi MR }
\providecommand{\MRhref}[2]{%
  \href{http://www.ams.org/mathscinet-getitem?mr=#1}{#2}
}
\providecommand{\href}[2]{#2}
\begin{thebibliography}{10}

\bibitem{MR2031121}
Michael~A. Bennett and Chris~M. Skinner, \emph{Ternary {D}iophantine equations
  via {G}alois representations and modular forms}, Canad. J. Math. \textbf{56}
  (2004), no.~1, 23--54.

\bibitem{MR2406491}
Nicolas Billerey, \emph{Formes homog\`enes de degr\'e 3 et puissances
  {$p$}-i\`emes}, J. Number Theory \textbf{128} (2008), no.~5, 1272--1294.

\bibitem{MR1863855}
Yu. Bilu, G.~Hanrot, and P.~M. Voutier, \emph{Existence of primitive divisors
  of {L}ucas and {L}ehmer numbers}, J. Reine Angew. Math. \textbf{539} (2001),
  75--122, With an appendix by M. Mignotte. \MR{1863855 (2002j:11027)}

\bibitem{MR1484478}
Wieb Bosma, John Cannon, and Catherine Playoust, \emph{The {M}agma algebra
  system. {I}. {T}he user language}, J. Symbolic Comput. \textbf{24} (1997),
  no.~3-4, 235--265.

\bibitem{MR1839918}
Christophe Breuil, Brian Conrad, Fred Diamond, and Richard Taylor, \emph{On the
  modularity of elliptic curves over {$\mathbf{Q}$}: wild 3-adic exercises}, J.
  Amer. Math. Soc. \textbf{14} (2001), no.~4, 843--939 (electronic).

\bibitem{MR2215137}
Yann Bugeaud, Maurice Mignotte, and Samir Siksek, \emph{Classical and modular
  approaches to exponential {D}iophantine equations. {I}. {F}ibonacci and
  {L}ucas perfect powers}, Ann. of Math. (2) \textbf{163} (2006), no.~3,
  969--1018.

\bibitem{CaldM}
Chris Caldwell, \emph{Mersenne primes: History, theorems and lists},
  \url{http://primes.utm. edu/mersenne/index.html}.

\bibitem{CaldF}
\bysame, \emph{The prime pages: Fibonacci prime}, \url{http://primes.utm.edu/
  glossary/page.php?sort=FibonacciPrime}.

\bibitem{MR866702}
D.~V. Chudnovsky and G.~V. Chudnovsky, \emph{Sequences of numbers generated by
  addition in formal groups and new primality and factorization tests}, Adv. in
  Appl. Math. \textbf{7} (1986), no.~4, 385--434. \MR{866702 (88h:11094)}

\bibitem{MR2312338}
Henri Cohen, \emph{Number theory. {V}ol. {II}. {A}nalytic and modern tools},
  Graduate Texts in Mathematics, vol. 240, Springer, New York, 2007.

\bibitem{MR2377127}
Gunther Cornelissen and Karim Zahidi, \emph{Elliptic divisibility sequences and
  undecidable problems about rational points}, J. Reine Angew. Math.
  \textbf{613} (2007), 1--33.

\bibitem{MR1628193}
J.~E. Cremona, \emph{Algorithms for modular elliptic curves}, Cambridge
  University Press, 1997.

\bibitem{Sanderthesis}
Sander~R. Dahmen, \emph{Classical and modular methods applied to {D}iophantine
  equations}, Ph.D. thesis, University of Utrecht, 2008,
  \url{http://igitur-archive.library.uu.nl/dissertations/2008-0820-200949/UUin%
dex.html}.

\bibitem{MR1348707}
Henri Darmon and Andrew Granville, \emph{On the equations {$z^m=F(x,y)$} and
  {$Ax^p+By^q=Cz^r$}}, Bull. London Math. Soc. \textbf{27} (1995), no.~6,
  513--543.

\bibitem{MR1468926}
Henri Darmon and Lo{\"{\i}}c Merel, \emph{Winding quotients and some variants
  of {F}ermat's last theorem}, J. Reine Angew. Math. \textbf{490} (1997),
  81--100.

\bibitem{MR2112196}
Fred Diamond and Jerry Shurman, \emph{A first course in modular forms},
  Graduate Texts in Mathematics, vol. 228, Springer-Verlag, New York, 2005.

\bibitem{MR2480276}
Kirsten Eisentr{\"a}ger and Graham Everest, \emph{Descent on elliptic curves
  and {H}ilbert's tenth problem}, Proc. Amer. Math. Soc. \textbf{137} (2009),
  no.~6, 1951--1959.

\bibitem{EES}
Kirsten Eisentr{\"a}ger, Graham Everest, and Alexandra Shlapentokh,
  \emph{Hilbert's {T}enth {P}roblem and {M}azur's {C}onjectures in
  {C}omplementary {S}ubrings of {N}umber {F}ields},
  \url{http://arxiv.org/abs/1012.4878}, 2010.

\bibitem{MR2486632}
Graham Everest, Patrick Ingram, and Shaun Stevens, \emph{Primitive divisors on
  twists of {F}ermat's cubic}, LMS J. Comput. Math. \textbf{12} (2009), 54--81.

\bibitem{MR2164113}
Graham Everest and Helen King, \emph{Prime powers in elliptic divisibility
  sequences}, Math. Comp. \textbf{74} (2005), no.~252, 2061--2071 (electronic).

\bibitem{MR2045409}
Graham Everest, Victor Miller, and Nelson Stephens, \emph{Primes generated by
  elliptic curves}, Proc. Amer. Math. Soc. \textbf{132} (2004), no.~4, 955--963
  (electronic).

\bibitem{EvOm}
Graham Everest, Ouamporn Phuksuwan, and Shaun Stevens, \emph{The uniform
  primality conjecture for the twisted fermat cubic},
  \url{http://arxiv.org/abs/1003.2131}, 2010.

\bibitem{MR2669714}
Bet{\"u}l Gezer and Osman Bizim, \emph{Squares in elliptic divisibility
  sequences}, Acta Arith. \textbf{144} (2010), no.~2, 125--134.

\bibitem{MR2301226}
Patrick Ingram, \emph{Elliptic divisibility sequences over certain curves}, J.
  Number Theory \textbf{123} (2007), no.~2, 473--486.

\bibitem{IngrSilv}
Patrick Ingram and Joseph~H. Silverman, \emph{Uniform estimates for primitive
  divisors in elliptic divisibility sequences}, to appear in a forthcoming
  memorial volume for Serge Lang, published by Springer-Verlag.

\bibitem{MR1979902}
Wilfrid Ivorra, \emph{Sur les \'equations {$x^p+2^\beta y^p=z^2$} et
  {$x^p+2^\beta y^p=2z^2$}}, Acta Arith. \textbf{108} (2003), no.~4, 327--338.
  \MR{1979902 (2004b:11036)}

\bibitem{MR1166121}
A.~Kraus and J.~Oesterl{\'e}, \emph{Sur une question de {B}. {M}azur}, Math.
  Ann. \textbf{293} (1992), no.~2, 259--275.

\bibitem{MR1611640}
Alain Kraus, \emph{Majorations effectives pour l'\'equation de {F}ermat
  g\'en\'eralis\'ee}, Canad. J. Math. \textbf{49} (1997), no.~6, 1139--1161.

\bibitem{Rey02}
Jonathan Reynolds, \emph{Perfect powers in elliptic divisibility sequences},
  \url{http://arxiv.org/abs/1101.2949}, 2011.

\bibitem{MR1047143}
K.~A. Ribet, \emph{On modular representations of {${\rm Gal}(\overline{\bf
  Q}/{\bf Q})$} arising from modular forms}, Invent. Math. \textbf{100} (1990),
  no.~2, 431--476.

\bibitem{MR1265566}
Kenneth~A. Ribet, \emph{Report on mod {$l$} representations of {${\rm
  Gal}(\overline{\bf Q}/{\bf Q})$}}, Motives ({S}eattle, {WA}, 1991), Proc.
  Sympos. Pure Math., vol.~55, Amer. Math. Soc., Providence, RI, 1994,
  pp.~639--676.

\bibitem{Shi01}
R.~Shipsey, \emph{Elliptic divisibility sequences}, Ph.D. thesis, Goldsmith's
  College (University of London), 2000,
  \url{http://homepages.gold.ac.uk/rachel/#PhD}.

\bibitem{MR2142239}
Samir Siksek, \emph{On the {D}iophantine equation {$x^2=y^p+2^kz^p$}}, J.
  Th\'eor. Nombres Bordeaux \textbf{15} (2003), no.~3, 839--846. \MR{2142239
  (2005m:11049)}

\bibitem{MR961918}
Joseph~H. Silverman, \emph{Wieferich's criterion and the {$abc$}-conjecture},
  J. Number Theory \textbf{30} (1988), no.~2, 226--237.

\bibitem{MR1312368}
\bysame, \emph{Advanced topics in the arithmetic of elliptic curves}, Graduate
  Texts in Mathematics, vol. 151, Springer-Verlag, New York, 1994.

\bibitem{MR2514094}
\bysame, \emph{The arithmetic of elliptic curves}, Graduate Texts in
  Mathematics, vol. 106, Springer, 2009.

\bibitem{MR1171452}
Joseph~H. Silverman and John Tate, \emph{Rational points on elliptic curves},
  Undergraduate Texts in Mathematics, Springer-Verlag, New York, 1992.

\bibitem{StangeL}
Katherine Stange and Kristin Lauter, \emph{The elliptic curve discrete
  logarithm problem and equivalent hard problems for elliptic divisibility
  sequences}, Selected Areas in Cryptography \textbf{5381} (2008), 309--327.

\bibitem{Marcothesis}
Marco Streng, \emph{Elliptic divisibility sequences with complex
  multiplication}, Master's thesis, Universiteit Utrecht, 2006,
  \url{http://www.warwick.ac.uk/~masjap/mthesis.pdf}.

\bibitem{MR1333036}
Richard Taylor and Andrew Wiles, \emph{Ring-theoretic properties of certain
  {H}ecke algebras}, Ann. of Math. (2) \textbf{141} (1995), no.~3, 553--572.

\bibitem{MR0023275}
Morgan Ward, \emph{Memoir on elliptic divisibility sequences}, Amer. J. Math.
  \textbf{70} (1948), 31--74.

\bibitem{MR1333035}
Andrew Wiles, \emph{Modular elliptic curves and {F}ermat's last theorem}, Ann.
  of Math. (2) \textbf{141} (1995), no.~3, 443--551.

\end{thebibliography}

\end{document}